\documentclass[a4paper,11pt,reqno]{amsart}

\allowdisplaybreaks

\usepackage{amsthm,amsmath,amsfonts,amssymb}
\usepackage{cases}
\usepackage{xfrac}
\usepackage{tikz}

\usepackage{xcolor}

\usepackage[
colorlinks=true,
urlcolor=teal,
citecolor=magenta,
linkcolor=teal,
linktocpage,
pdfpagelabels,
bookmarksnumbered,
bookmarksopen]
{hyperref}
\usepackage[
capitalize, 
nameinlink,
noabbrev]
{cleveref}

\def\de{\mathrm{d}}
\def\Fp{F_p} 

\def\step#1#2{\par\noindent{\underline{\it Step~#1.}}\emph{ #2}\\}
\def\bal{\begin{aligned}}
\def\eal{\end{aligned}}
\def\R{\mathbb{R}}
\def\N{\mathbb{N}}
\font\script=rsfs10 at 11pt
\def\eps{\varepsilon}
\def\H{{\mbox{\script H}\,\,}}
\def\Chi#1{\hbox{{\large $\chi$}{\Large $_{_{#1}}$}}}

\DeclareMathOperator{\diam}{diam}
\DeclareMathOperator{\essinf}{ess\,inf}

\theoremstyle{plain}
\newtheorem{thm}{Theorem}[section]
\newtheorem*{thm*}{Theorem}
\newtheorem{lem}[thm]{Lemma}

\newtheorem{cor}[thm]{Corollary}

\theoremstyle{definition}

\theoremstyle{remark}
\newtheorem{rem}[thm]{Remark}

\numberwithin{equation}{section}

\title[Cylindrical estimates for the Cheeger constant]{Cylindrical estimates for the\\Cheeger constant and applications}

\author[A.~Pratelli]{Aldo Pratelli}
\address[Aldo Pratelli]{Dipartimento di Matematica, Universit\`a di Pisa, largo B.~Pontecorvo 5, IT--56127 Pisa}
\email{aldo.pratelli@unipi.it}
\author[G.~Saracco]{Giorgio Saracco}
\address[Giorgio Saracco]{Dipartimento di Matematica e Informatica ``Ulisse Dini'', Universit\`a di Firenze, viale Morgagni 67/A, IT--50134 Firenze}
\email{giorgio.saracco@unifi.it}%

\thanks{The first author has been partially supported by the project PRIN Geometric Evolution Problems and Shape Optimization (GEPSO), 2022E9CF89. The second author is member of INdAM--GNAMPA and has been partially supported by the INdAM--GNAMPA 2024 Project \textit{Ottimizzazione e disuguaglianze funzionali per problemi geometrico--spettrali locali e nonlocali}, codice CUP\_E53C23001670001.}

\subjclass[2020]{Primary: 49Q10. Secondary: 49R05, 35P15}

\keywords{Cheeger constant, convex sets, shape optimization, cylinders, asymptotic estimates}

\begin{document}

\begin{abstract}
We prove a lower bound for the Cheeger constant of a cylinder $\Omega\times (0,L)$, where $\Omega$ is an open and bounded set. As a consequence, we obtain existence of minimizers for the shape functional defined as the ratio between the first Dirichlet eigenvalue of the $p$-Laplacian and the $p$-th power of the Cheeger constant, within the class of bounded convex sets in any $\R^N$. This positively solves open conjectures raised by Parini (\emph{J.\ Convex Anal.}\ (2017)) and by Briani--Buttazzo--Prinari (\emph{Ann.\ Mat.\ Pura Appl.}\ (2023)).
\end{abstract}


\maketitle

\section{Introduction}

The Cheeger constant of an open, bounded set $\Omega\subseteq \R^N$ is defined as
\[
h(\Omega) :=\inf \left\{\,\frac{P(E)}{|E|}\,:\, E\subseteq \Omega,\, |E|>0 \,\right\}\,.
\]
The constant owes its name to Cheeger, who in~\cite{Che70} used its Riemannian counterpart to provide a lower bound to the first eigenvalue of the Dirichlet Laplace--Beltrami operator.

We recall that, in the Euclidean setting, the first eigenvalue of the $p$-Dirichlet Laplacian ($p>1$), for an open, bounded set $\Omega\subseteq \R^N$ is variationally characterized as
\[
\lambda_p(\Omega) := \inf \left\{\,\int_\Omega |\nabla u|^p\,\mathrm{d}x\,:\, u\in W^{1,p}_0(\Omega)\, \text{and}\, \|u\|_p=1 \,\right\}\,,
\]
where $W^{1,p}_0(\Omega)$ is taken as the closure of $C^\infty_c(\Omega)$ with respect to the Sobolev norm $\|\cdot\|_p + \|\nabla(\cdot)\|_p$. Using this notation, one has
\begin{equation}\label{eq:mela}
\lambda_p(\Omega) \ge \left(\frac{h(\Omega)}{p}\right)^p\,,
\end{equation}
refer to~\cite[Theorem~3.1]{Avi97} and~\cite[Appendix]{LW97}. Further, assuming $\Omega$ to be regular enough (Lipschitz suffices), $\lim_{p \to 1^+}\lambda_p(\Omega)=h(\Omega)$ holds, as first proved in~\cite{KF03}. Inequality~\eqref{eq:mela} is very robust, as noticed even earlier than Cheeger in the linear case $p=2$ by Maz'ya~\cite{Maz62,Maz62-2}, and can be extended to the abstract setting of perimeter-(topological) measure spaces for Borel sets, see, e.g.,~\cite{FPSS22} and the references therein. For an overview of the Cheeger problem, we refer the interested reader to the surveys~\cite{Leo15, Par11}. Rearranging~\eqref{eq:mela}, one obtains
\begin{equation}\label{eq:cheeger_bound}
\Fp[\Omega] := \frac{\lambda^{\sfrac1p}_p(\Omega)}{h(\Omega)} \ge \frac{1}{p}\,.
\end{equation}
Hence, one has a lower bound to the spectral operator $\Fp[\,\cdot\,]$. It is therefore natural to wonder whether this bound is attained and in general if the minimization of $\Fp[\,\cdot\,]$ has solutions in some suitable class of subsets of $\mathbb{R}^N$. Here, we shall focus on the class of convex subsets of $\mathbb{R}^N$.

In the convex case, the bound is known not to be sharp, as tighter bounds have been proved in some special cases. Namely, for $p=2$ and general $N$, it holds
\begin{equation}\label{eq:parini_imp}
F_2[\Omega] \geq \frac{\pi}{2N}\,,
\end{equation}
as proved in~\cite[Proposition~5.1]{Par17} for $N=2$ and later noticed to hold for any $N$ in~\cite[Introduction]{Fto21}. This provides a tighter bound in dimension $N=2,3$. In dimension $N=2$, a further refinement was given in~\cite[Theorem.~1.1]{Fto21}, where it was shown that
\[
F_2[\Omega] \ge \frac{\pi j_{01}}{2j_{01}+\pi}\,,
\]
where $j_{01}$ denotes the first zero of the first Bessel function. An inequality similar in spirit to~\eqref{eq:parini_imp} holds for general $p$, refer to~\cite[Proposition~3.2]{BBP22}, namely,
\[
F_p[\Omega] \ge \max\left\{\,\frac 1p ; \frac{\pi(p-1)^{\frac 1p}}{N p \sin\left(\frac \pi p\right)}\,\right\}\,.
\]
Hence, for dimensions $N\le \bar{N}(p)$ one gets a tighter bound than the one in~\eqref{eq:cheeger_bound}. As a side result of this paper, we shall obtain that~\eqref{eq:cheeger_bound} is not sharp in any dimension $N$ and for any $p$. For the sake of completeness, we mention that in the recent preprint~\cite[Remark~4.2]{CP24} it was provided an improvement (depending on $N$ and $p$) on~\eqref{eq:mela} under the lone openness of $\Omega$, so that, one gets again that~\eqref{eq:cheeger_bound} is not sharp for dimensions $N$ smaller than some $\bar N(p)$ when restricting competitors to open sets. 

Regarding existence of minimizers, the first step in this direction was made in~\cite[Proposition~5.2]{Par17}, where Parini proved existence among the class of convex subsets of the plane $\mathbb{R}^2$, in the case $p=2$. Later on, Ftouhi~\cite[Theorem~1.2]{Fto21} provided a different proof of existence, along with a sufficient criterion to determine whether minimizers among convex subsets of $\mathbb{R}^{N}$ exist. Recently, Buttazzo, Briani and Prinari extended this criterion to general $p$, and proved existence for $p\ge 2$ among convex subsets of the plane $\mathbb{R}^2$, see~\cite[Theorem~3.6]{BBP22}. To the best of our knowledge, the existence of minimizers among open subsets of $\mathbb{R}^N$ is completely open (except for $N=1$ which is trival, see~\cite[Proposition~2.1]{BBP22}).

In this paper, by exploiting the above-mentioned criteria, we are able to show existence of minimizers of $\Fp[\,\cdot\,]$ among convex subsets of $\mathbb{R}^N$, for any $N\ge 2$ and $p>1$, see \cref{thm:ex_lp/h^p}. The proof relies on some cylindrical estimates on the Cheeger constant. The key one, see \cref{thm:stime_h}, is the following: given an open and bounded subset $\Omega$ of $\mathbb{R}^N$, not necessarily convex, and given $L\ge 1$, we show that
\begin{equation}\label{eq:estimate_h}
h(\Omega) + \frac{c(\Omega)}{L} \le h(\Omega_L) < h(\Omega) + \frac{2}{L}\,,
\end{equation}
where $\Omega_L:=\Omega\times (0,L)$ and $c(\Omega)>0$ is a constant depending only on $\Omega$. In other words, we can estimate both from above and from below the Cheeger constant of the $(N+1)$-dimensional cylinder $\Omega_L$ with the Cheeger constant of its cross-section plus a non-zero term that goes like the inverse of the height of the cylinder. The proof of \cref{thm:ex_lp/h^p} essentially follows from the
combination of~\eqref{eq:estimate_h} and the criteria proved in~\cite{BBP22,Fto21}.

The paper is organized as follows: in \cref{sec:estimates}, we prove the cylindrical estimates, refer to \cref{thm:stime_h}; in \cref{sec:appl}, we use them to obtain \cref{thm:ex_lp/h^p}.

\subsection*{Acknowledgements} The authors wish to thank Giuseppe Buttazzo for making them aware of the problem and the two referees for their valuable comments which, in particular, pointed us how to remove the assumption $L\ge 1$ of~\cref{thm:stime_h}, up to assuming the convexity of the cross-section, see \cref{rem:lge1}.

\section{Estimates}\label{sec:estimates}

This section is devoted to show the key estimate~\eqref{eq:estimate_h} for cylinders. First of all, we recall some standard notation, see for instance~\cite{AFP00book}. Given any $M\in \mathbb{N}$, we let $\H^M$ be the standard  $M$-dimensional Hausdorff measure. Given a set of finite perimeter $A\subset \mathbb{R}^{M+1}$, we denote by $\partial^*A$ its reduced boundary, and we recall that $P(A; B)$, the perimeter of $A$ relative to a Borel set $B$, equals $\H^{M}(\partial^*A \cap B)$. Finally, if $B=\mathbb{R}^{M+1}$, one writes $P(A)$ in place of $P(A;\mathbb{R}^{M+1})$.

Second, we fix a quick notation that will be used through the rest of the paper. Given any set $\Omega\subseteq\R^N$, for any $L>0$ the cylinder $\Omega\times(0,L)\subseteq \R^{N+1}$ will be denoted by $\Omega_L$. Moreover, for any subset $C\subseteq \Omega_L$ and $t\in [0,L]$ we will denote by $C_t$ the horizontal section of $C$ at height $t$, that is,
\[
C_t :=
\begin{cases}
C\cap (\Omega \times\{t\}), \qquad &t\in (0,L)\,,\\
\partial C \cap  (\Omega \times\{t\}), &t\in \{0,L\}\,.
\end{cases}
\]
Moreover, with some abuse of notation we shall write $C_t\subseteq \Omega$, in place of $\Pi_N(C_t) \subseteq \Omega$, being $\Pi_N$ the restriction to the first $N$ coordinates. Notice also that, by the well-known Vol'pert Theorem (see for instance~\cite[Theorem~6.2]{FMP08}), if $C$ is a set of finite perimeter in $\Omega_L$ then for a.e.\ $t\in [0,L]$ the section $C_t$ is a well-defined set with finite perimeter in $\R^N$, where, from now on, by \emph{well-defined} we mean uniquely defined up to sets of zero $\H^N$-measure.

The main result of the section reads as follows.
\begin{thm}\label{thm:stime_h}
Let $\Omega\subseteq\R^N$ be open and bounded. There exists a constant $c(\Omega)>0$ such that for any $L \ge 1$ one has
\begin{equation}\tag{\ref{eq:estimate_h}}
h(\Omega) + \frac{c(\Omega)}{L} \le h(\Omega_L) \le h(\Omega) + \frac{2}{L}\,.
\end{equation}
\end{thm}

Before proving this result, we need to show the following preliminary lemma. As recalled, by Vol'pert Theorem, if $D\subset \Omega_L$ is a set of finite perimeter, then almost every section $D_t$ is well-defined. The lemma states that for those values $t\in [0,L]$ for which the section is not well-defined, one can however ``identify'' it either as the unique limit of sections above it (i.e., for $s\searrow t$) or as the unique limit (possibly different from the previous one) of sections below it (i.e., for $s\nearrow t$).

\begin{lem}\label{supercont}
Let $D$ be any set of finite perimeter contained in the cylinder $\Omega_L$. Then, for every $0\leq t<L$, there exists a (unique up to negligible sets) set $D^+_t\subseteq\Omega$ such that
\begin{equation}\label{newcla}
\lim_{s\searrow t} \big| D_s \triangle D^+_t \big| = 0\,.
\end{equation}
Analogously, for every $0<t\leq L$ there exists a (unique up to negligible sets) set $D^-_t\subseteq\Omega$ such that 
\[
\lim_{s\nearrow t} \big| D_s \triangle D^-_t \big| = 0\,.
\]
Moreover, $D_t^+=D_t^-$ ($\H^N$-a.e.) for $t\in (0,L)$ except at most countably many $t$.
\end{lem}
\begin{proof}
Let us fix $0\leq t <L$, and let $s_j\searrow t$ be any monotone decreasing sequence converging to $t$ and such that the section $D_{s_j}$ is well-defined for every $j$. Notice that
\[
P\Big(D; \Omega\times \big( s_{j+1}, s_j\big)\Big)= \H^N\Big(\partial^* D \cap \big(\Omega\times( s_{j+1}, s_j)\big)\Big)\,,
\]
and that, calling $\Pi_N:\R^{N+1}\to\R^N$ the projection on the first $N$ coordinates, we have
\[
\Pi_N\Big(\partial^* D \cap \big(\Omega\times( s_{j+1}, s_j)\big)\Big) \supseteq D_{s_{j+1}}\triangle D_{s_j}\,.
\]
Since the projection is $1$-Lipschitz, we deduce that
\begin{equation}\label{exrem}
P\Big(D; \Omega\times \big( s_{j+1}, s_j\big)\Big) \geq \big| D_{s_{j+1}}\triangle D_{s_j}  \big|\,.
\end{equation}
As a consequence, taking the sum on $j$, and recalling that $D$ has finite perimeter, we have that
\begin{equation}\label{eq:finite_sum}
+\infty > P(D) \geq \sum_{j\in\N} P\Big(D; \Omega\times \big( s_{j+1}, s_j\big)\Big) 
\geq \sum_{j\in\N} \big| D_{s_{j+1}}\triangle D_{s_j}  \big|\,.
\end{equation}
We now extract a subsequence $\{\sigma_n\}$ of $\{s_j\}$ as follows. By the finiteness of the sum in~\eqref{eq:finite_sum}, for any integer $n\ge 1$ there exists an index $j_n$ such that the tail of the series is bounded from above as
\begin{equation}\label{eq:stima_tail}
 \sum_{j\ge j_n} \big| D_{s_{j+1}}\triangle D_{s_j}  \big| < \frac{1}{2^n}\,.
\end{equation}
Moreover, we can also assume that $\{j_n\}$ is strictly monotone, and that
\begin{equation}\label{aspt}
\text{$j_n$ and $n$ have the same parity for each $n$.}
\end{equation}
We let the sequence $\{\sigma_n\} := \{s_{j_n}\}$, and we notice that, owing to~\eqref{eq:stima_tail} and the set-wise triangular inequality $A\triangle B \subseteq A\triangle C \cup C\triangle B$, the sections of $D$ individuated by the sequence of heights $\{\sigma_n\}$ satisfy
\begin{equation}\label{eq:control_section_subs}
\big| D_{\sigma_{n+1}}\triangle D_{\sigma_{n}}  \big| \le \sum_{j=j_n}^{j_{n+1}-1} \big| D_{s_{j+1}}\triangle D_{s_j}  \big|  < \frac 1{2^n}\,,\qquad \forall\, n\in\N\,.
\end{equation}
Calling now, for every $k\in\N$,
\begin{align*}
F_k:= \bigcap_{n\geq k}  D_{\sigma_n}\,, && G_k:= \bigcup_{n\geq k}  D_{\sigma_n}\,, 
\end{align*}
we trivially have that $F_k \subseteq D_{\sigma_k} \subseteq G_k$. Moreover, one can easily prove the set equality
\[
G_k \setminus F_k = \bigcup_{n\ge k} \left(D_{\sigma_{n+1}} \triangle D_{\sigma_{n}}\right)\,,
\]
which combined with~\eqref{eq:control_section_subs} yields the estimate $\big| G_k \setminus F_k\big| \leq 2^{1-k}$. Notice that the set equality
\begin{equation}\label{eq:defD+}
\bigcup_{k\in {\mathbb{N}}} F_k = \bigcap_{k\in\N} G_k\,,
\end{equation}
holds up to sets of measure zero, since the sets $F_k$ are monotonically increasing, the sets $G_k$ monotonically decreasing, $F_k\subset G_k$, and $\big| G_k \setminus F_k\big| \leq 2^{1-k}$. We then let $D^+_t$ be the set in~\eqref{eq:defD+}, which is well-defined. We have that
\begin{equation}\begin{split}\label{weaker}
\lim_{k\to\infty} \big| D_{\sigma_k} \triangle D^+_t \big| &= \lim_{k\to\infty} \big| D_{\sigma_k} \setminus D^+_t \big| +\big| D^+_t\setminus D_{\sigma_k} \big| \\
&\le \lim_{k\to\infty} \big| D_{\sigma_k} \setminus F_k \big| + \big| G_k \setminus D_{\sigma_k} \big|
 = \lim_{k\to\infty} \big| G_k \setminus F_k \big| =0\,,
\end{split}\end{equation}
where the second-to-last equality comes from the fact that $F_k\subseteq D_{\sigma_k}\subseteq G_k$. This property is in principle weaker than~\eqref{newcla}, because the set $D^+_t$ might depend on the particular choice of the sequence $\{s_j\}$ and of the subsequence $\{\sigma_n\}$. Therefore, it suffices to show that any choice of the sequence and of the relative subsequence yields the same limit set $D^+_t$. Consider two monotone sequences $s_j'\searrow t$ and $s_j''\searrow t$ satisfying
\[
\lim_{j\to\infty} \big| D_{s'_j} \triangle D' \big|=\lim_{j\to\infty} \big| D_{s''_j} \triangle D'' \big|=0
\]
for two different sets $D',\, D''$, and define a monotone sequence $s_j\searrow t$ in such a way that $\{s_{2j}\}$ is a subsequence of $\{s_j'\}$ and $\{s_{2j+1}\}$ is a subsequence of $\{s_j''\}$, which is clearly possible. Reasoning as before, one can find a subsequence $\{\sigma_n\}$ of $\{s_j\}$ and a set $D^+_t$ such that~\eqref{weaker} holds, but by construction and thanks to~\eqref{aspt} we deduce that $D^+_t$ must coincide both with $D'$ and $D''$, which is impossible since $D'$ and $D''$ are different. The contradiction shows that actually $D^+_t$ does not depend on the sequence, so that we have obtained~\eqref{newcla}.

The existence of the set $D^-_t$ can be of course obtained exactly in the same way. To conclude the proof, we only have to check that there can be at most countably many $t\in (0,L)$ for which $D_t^+\neq D_t^-$ in a measure-theoretic sense, that is, $|D_t^+\triangle D_t^-|>0$. By the same projection argument around~\eqref{exrem}, looking at the perimeter of $D$ in the strip $\mathbb{R}^N\times(t-\eps, t+\eps)$, and then letting $\eps$ to zero, we get that the perimeter of $D$ relative to the hyperplane $\mathbb{R}^N\times\{t\}$ is at least $|D^+_t \triangle D^-_t|$. Since $D$ has finite perimeter, this can happen at most countably many times.
\end{proof}

We are now ready to present the proof of the main estimate of this section.

\begin{proof}[Proof of \cref{thm:stime_h}]
The proof will be divided in some steps.
\step{I}{A weaker inequality.}
In this first step, we prove an estimate which is weaker than~\eqref{eq:estimate_h}, namely,
\begin{equation}\label{eq:weaker}
h(\Omega) \leq h(\Omega_L) \le  h(\Omega) + \frac{2}{L}\,.
\end{equation}

We start by proving the upper bound. We let $F$ be a Cheeger set for $\Omega$, that is, a set realizing the infimum in the definition of the Cheeger constant, and whose existence is ensured by the boundedness of $\Omega$, refer to~\cite[Proposition~3.5(iii)]{Leo15}. It is then enough to notice that
\[
h(\Omega_L) \leq \frac{P\big(F \times [0,L]\big)}{\big| F\times [0,L]\big|}
=\frac{L P(F) + 2 |F|}{L |F|}
= \frac{P(F)}{|F|} + \frac 2L
= h(\Omega) + \frac 2L\,,
\]
so that the right inequality in~\eqref{eq:weaker} is proved, and so it also is the right one in~\eqref{eq:estimate_h}.

We now turn our attention to the lower bound. This estimate is very easy to prove and it is not needed to prove the stronger one~\eqref{eq:estimate_h}, yet it contains the basic ideas we shall exploit to prove our main result. Let $D\subseteq\Omega_L$ be any set of finite perimeter. Then, as previously mentioned, by Vol'pert Theorem the section $D_t$ is a well-defined set with finite perimeter in $\R^N$ for a.e.\ $t\in [0,L]$. Furthermore, one has that $(\partial^* D)_t = \partial^* D_t$ ($\H^{N-1}$-a.e.) for a.e.\ $t$. Trivially, since $D$ has finite perimeter in $\mathbb{R}^{N+1}$, its reduced boundary $\partial^* D$ is $\H^{N}$-rectifiable~\cite[Theorem~3.59]{AFP00book}. Thus, by the coarea formula for rectifiable sets (see~\cite[Theorem~2.93 and Remark~2.94]{AFP00book}), we have that
\begin{align}
P(D) = \int_{\partial^* D} 1\,\mathrm{d}\H^{N} &\ge \int_{\partial^* D} \sqrt{1- \|\nu_D \cdot e_{N+1}\|^2}\,\mathrm{d}\H^{N}
\nonumber
\\
&
= \int_{\mathbb{R}}\int_{\partial^* D_t} 1\,\mathrm{d}\H^{N-1}\, \mathrm{d}t
= \int_{0}^L\int_{\partial^* D_t} 1\,\mathrm{d}\H^{N-1}\, \mathrm{d}t
= \int_0^L P(D_t)\,,
\label{eq:proiezione}
\end{align}
where by $P(D_t)$ we denote the perimeter of $D_t$ in $\R^N$, that is, $\H^{N-1}(\partial^* D_t)$. Since $D_t\subseteq\Omega$, we have $P(D_t)\geq h(\Omega) |D_t|$, where by $|D_t|$ we denote the measure of $D_t$ in $\R^N$, that is, $\H^N(D_t)$. Therefore,
\begin{equation}\label{eq:sloppy}
P(D) \geq \int_{0}^L P(D_t)\,\de t \geq \int_{0}^L h(\Omega) |D_t|\,\de t = h(\Omega) \int_{0}^L |D_t|\, \de t = h(\Omega) |D|\,,
\end{equation}
thus $P(D)/|D| \geq h(\Omega)$ for every $D\subseteq \Omega_L$ and then the left inequality in~\eqref{eq:weaker} is proved. 

We remark that the inequality in~\eqref{eq:proiezione} remains true adding on the right the measure of the section $D_0^+$ given by \cref{supercont}. This simple observation will be the starting point of the proof of the stronger inequality.

\step{II}{The ``minimal volume'' $v$, the gap $\eps$, the height $\tau$ and the volume up to $\tau$, $V$.}
This step is devoted to define two positive quantities $v$ and $\eps$, which depend only on $\Omega$, and two other quantities $\tau$ and $V$ that depend on the choice of one Cheeger set of the cylinder $\Omega_L$. 

As recalled in the previous step, in view of the boundedness of $\Omega$, there exist Cheeger sets for $\Omega$. In general there might be several different Cheeger sets, but their measure cannot be too small. In particular, given any Cheeger set $C$ of $\Omega$, one has
\begin{equation}\label{eq:lower_bound_meas}
|C| \geq \omega_N\left(\frac{N}{h(\Omega)}\right)^N\,,
\end{equation}
where $\omega_N$ is the $N$-dimensional Lebesgue measure of a unit-radius ball in $\mathbb{R}^N$, refer to~\cite[Proposition~3.5(v)]{Leo15}. We now let
\[
v := \inf \Big\{ |C|\,:\, \hbox{$C$ is a Cheeger set for $\Omega$}\Big\}\,,
\]
and we have that $v=v(\Omega)>0$ in view of~\eqref{eq:lower_bound_meas}. We then let
\[
\eps:= \inf \bigg\{ \frac{P(E)}{|E|},\, E\subseteq \Omega,\, |E| \leq v/2\bigg\} - h(\Omega)\,.
\]
It is clear that $\eps=\eps(v(\Omega), h(\Omega))$ and thus, definitively, $\eps=\eps(\Omega)$. It is simple to notice that $\eps>0$. Indeed, let $E$ be any subset of $\Omega$ with $|E|\leq v/2$. By definition of $v$, $E$ is not a Cheeger set for $\Omega$, and then
\[
\frac{P(E)}{|E|} > h(\Omega)\,.
\]
This only shows that $\eps\geq 0$. Argue now by contradiction and assume the existence of a sequence of sets $E_j\subseteq \Omega$ such that $|E_j|\leq v/2$ for every $j$, and
\begin{equation}\label{eq:conv_h_imp}
\frac{P(E_j)}{|E_j|} \to  h(\Omega)\,.
\end{equation}
The sequence of the characteristic functions $\Chi{E_j}$ is then bounded in $BV(\Omega)$, so that, up to a subsequence, we can assume that $\Chi{E_j}$ is weakly-star converging in $BV(\Omega)$ to some function $\varphi$. However, since $\Omega$ is bounded the convergence is strong in $L^1$, thus $\varphi$ is the characteristic function of some set $E_\infty$ with $|E_\infty|>0$, as otherwise~\eqref{eq:conv_h_imp} would be easily contradicted by using the isoperimetric inequality. Indeed, were $|E_\infty|=0$, we would have that $|E_j|\to 0$. Hence, calling $B_j$ any ball with the same measure of $E_j$, and $r_j$ its radius, we would have
\[
\frac{P(E_j)}{|E_j|} \ge \frac{P(B_j)}{|B_j|} = \frac{N \omega_N}{r_j} \to +\infty,
\]
because $r_j$ goes to zero as $|B_j|\to 0$, against~\eqref{eq:conv_h_imp}. Then, the lower semicontinuity of the perimeter implies that $E_\infty$ is a Cheeger set for $\Omega$ with measure less than $v/2$, which is impossible by definition of $v$. Hence, we conclude that $\eps>0$.

Let us now fix a Cheeger set $C^*$ of $\Omega_L$. Notice that for almost every $t\in [0,L]$, $(C^*)^+_t$ and $(C^*)^-_t$ coincide by \cref{supercont}, and by Vol'pert Theorem they also coincide with $C^*_t$. We define the two following quantities
\begin{align*}
\tau := \essinf \big\{ t \in [0,L],\, |C^*_t|\geq v/2\big\}\,,  &&
V := \Big| C^* \cap \big(\Omega\times [0,\tau]\big)\Big|\,,
\end{align*}
which depend also on the choice of the Cheeger set $C^*$. Nevertheless, we will be able to check the validity of~\eqref{eq:estimate_h}, independently from such a choice, exhibiting a constant $c$ depending only on the minimal volume $v$, the gap $\eps$, the measure $|\Omega|$, and the Cheeger constant $h(\Omega)$ and thus, definitively only on the set $\Omega$. 

\step{III}{The proof of the inequality.}
We shall now refine~\eqref{eq:sloppy}, arguing in different ways depending on the measure of the upper section of $C^*$ at height zero, which for the sake of convenience we denote by $C^*_0$ in place of $(C^*)^+_0$.

\emph{(i) The case $|C^*_0|> v/4$.}
We first assume that the bottom section of $C^*$ is not too small, namely, that $|C^*_0|>v/4$. Then, by refining~\eqref{eq:proiezione} adding, as already discussed, the extra term $ |C^*_0|$, and recalling that $P(C^*_t)\geq h(\Omega)|C^*_t|$ for all $t\in [0,L]$, we readily have
\begin{equation}\label{est1}
h(\Omega_L) = \frac{P(C^*)}{|C^*|} 
\geq \frac{\bal \int_{0}^L P(C^*_t)\,\de t + |C^*_0|\eal}{\bal\int_{0}^L |C^*_t|\, \de t \eal}
\geq \frac{\bal \int_{0}^L h(\Omega) |C^*_t|\,\de t + |C^*_0|\eal}{\bal\int_{0}^L |C^*_t|\, \de t \eal}
\geq h(\Omega) + \frac v {4|\Omega| L}\,,
\end{equation}
so the required estimate is obtained in this case.

\emph{(ii) The case $|C^*_0|\leq v/4$: when $V> v(8h(\Omega))^{-1}$.}
Notice that, by construction and by Step~II, $P(C^*_t)\geq (h(\Omega)+\eps) |C^*_t|$ for almost every $0\leq t \leq \tau$, while $P(C^*_t)\geq h(\Omega)|C^*_t|$ for almost every $\tau \leq t \leq L$. Therefore, again by integration we get
\begin{equation}\label{est2}\begin{split}
h(\Omega_L) 
&= \frac{P(C^*)}{|C^*|} 
\geq \frac{\bal \int_{0}^\tau P(C^*_t)\,\de t + \int_{\tau}^L P(C^*_t)\,\de t\eal}{|C^*|}\\
&\geq \frac{\bal \int_{0}^L h(\Omega) |C^*_t|\,\de t +  \int_{0}^\tau \eps  |C^*_t|\,\de t\eal}{|C^*|}
=h(\Omega) + \frac{\eps V}{|C^*|} > h(\Omega) + \frac{\eps v}{8 h(\Omega) |\Omega| L}\,,
\end{split}\end{equation}
and then the required estimate is obtained also in this case.

\emph{(iii) The case $|C^*_0|\leq v/4$: when $V\leq  v(8h(\Omega))^{-1}$ and $\tau < L$.}
We now assume that $\tau < L$, so that by definition of $\tau$ there is a sequence of $t_n\searrow \tau$ such that $C^*_{t_n}$ is well-defined and $|C^*_{t_n}|\ge \sfrac v2$. Arguing by projection as in the proof of \cref{supercont} for any $n$, we have
\[
P\Big( C^* ; \big( \Omega\times (0,t_n)\big)\Big) \geq \big| |C^*_{t_n}| - |C^*_0|\big| \ge \frac v2 - |C^*_0| \geq \frac v 4\,.
\]
Therefore, arguing as in the previous steps, we get
\[
\begin{split}
h(\Omega_L) 
= 
\frac{P(C^*)}{|C^*|}
&
\geq 
\frac{P\Big( C^* ; \big( \Omega\times (0,t_n)\big)\Big)+ \bal \int_{t_n}^L P(C^*_t)\,\de t\eal}{|C^*|}
\\
&
\geq 
\frac{\bal\frac v 4+ \int_{t_n}^L h(\Omega)|C^*_t|\,\de t\eal}{|C^*|}
= 
\frac{\bal\frac v 4+ h(\Omega)\left(|C^*|-V - \int_{\tau}^{t_n}|C^*_t|\,\de t \right) \eal}{|C^*|}\,.
\end{split}
\]
Letting now $n\to +\infty$, we have
\begin{equation}\label{est3}
h(\Omega_L) 
\ge
h(\Omega) + \frac {v-4 h(\Omega) V}{4 |C^*|} 
\geq  h(\Omega) + \frac v{8 |C^*|} 
\geq  h(\Omega) + \frac v{8 |\Omega| L} \,,
\end{equation}
so the required estimate is obtained also in this case.

\emph{(iv) The case $|C^*_0|\leq v/4$: when $V\leq v(8h(\Omega))^{-1}$ and $\tau = L$.}
Being $\tau=L$, one has $|C^*_t|< v/2$ for almost every $t\in [0,L]$. In this case, the estimate $P(C^*_t) \geq \big(h(\Omega)+\eps\big) |C^*_t|$ is true for almost every $0\leq t\leq L$, and then arguing as usual this time we get, using also that $L\geq 1$,
\begin{equation}\label{est4}
h(\Omega_L) \geq h(\Omega) + \eps \geq h(\Omega)+ \frac \eps L\,.
\end{equation}

We are now ready to conclude. By putting together~\eqref{est1}, \eqref{est2}, \eqref{est3} and~\eqref{est4}, the claim follows by choosing $c(\Omega)$ as
\[
c(\Omega) = \min \bigg\{\frac {v(\Omega)} {8|\Omega|} \,; \frac{ {\eps(\Omega)}  {v(\Omega)}}{8 h(\Omega) |\Omega|} \,;  {\eps(\Omega)}\bigg\} >0\,,
\]
which only depends on $\Omega$ and not on the choice of the Cheeger set $C^*$.
\end{proof}

\begin{rem}[The assumption $L\ge 1$]\label{rem:lge1.1} Notice that the assumption $L\ge 1$ is used only in Step III~(iv) in order to obtain the second inequality in~\eqref{est4}. In particular, fixed any $\bar L>0$, one gets
\[
h(\Omega_L) \ge h(\Omega) + \frac{\varepsilon \bar L}{L}, \qquad \text{for any $L\ge \bar L$},
\]
so that, up to possibly changing the constant $c(\Omega)$, \cref{thm:stime_h} holds for all $L\ge \bar L$ for some given positive $\bar L$.
\end{rem}

\begin{rem}[The assumption $L\ge 1$, in the convex case]\label{rem:lge1}
Assuming the cross-section $\Omega$ to be convex, one can relax the assumption all the way down to $L>0$. In order to prove it, it is enough to show that the function $f:(0, +\infty)\to \mathbb{R}$ defined as $f: L\mapsto L(h(\Omega_L)-h(\Omega))$ is well-detached from zero in a sufficiently small right neighborhood of the origin $(0, \bar L)$, owing also to \cref{rem:lge1.1}. Thanks to the reverse Cheeger inequality~\cite{Par17} (refer also to~\cite[Remark~1.1]{Bra18}), to the equality $\lambda_2(\Omega_L) = \lambda_2(\Omega) + \frac{\pi^2}{L^2}$ (see~\cite[Lemma~2.4]{BBP22}), and exploiting the upper bound in~\eqref{eq:weaker}, we have
\begin{equation}\label{eq:extL0}
\frac{\pi^2}{4} > \frac{\lambda_2(\Omega_L)}{h(\Omega_L)^2} \ge \frac{\lambda_2(\Omega) + \frac{\pi^2}{L^2}}{(h(\Omega) + \frac 2L)^2} \xrightarrow[L\to 0^+]{} \frac{\pi^2}{4}\,.
\end{equation}
Therefore, $h(\Omega_L)$ behaves as $2/L$ as $L$ approaches $0$. Thus, $\lim_{L\to 0^+}f(L) =2$, hence $f$ is well-detached from zero when sufficiently close to the origin. The convexity of the cross-section is required because we make use of the reverse Cheeger inequality which corresponds to the leftmost inequality in~\eqref{eq:extL0}.
\end{rem}

\begin{rem}[Boundedness assumption]
The hypothesis that $\Omega$ is bounded in \cref{thm:stime_h} can be slightly relaxed. A close inspection to the proof highlights how this is used only to ensure the existence of isoperimetric sets within $\Omega$ (and, in particular, of Cheeger sets). Thus, one could drop it and require that $\Omega$ supports the compact embedding $BV(\Omega) \hookrightarrow L^1(\Omega)$ (i.e., it has finite measure and supports a relative isoperimetric inequality, see~\cite[Section~9.1.7]{Maz11book}). 
\end{rem}

\begin{rem}
In dimension $2$, there is only one kind of (connected) cylinder: rectangles. In this case, the constant is well-known and there is a formula to compute it depending only on the length of the sides of the rectangle, see the discussion after~\cite[Theorem~3]{KL06} together with the correction done in~\cite[Open problem~1]{Kaw16}. Our estimates are obviously consistent with such a formula. In the planar setting, similar estimates have been proved for ``strips'' ($2$d waveguides), that one can think of as bended rectangles, refer to~\cite[Theorem~3.2]{KP11} and also to~\cite[Theorem~3.2]{LP16}. Notice that the case of $2$-dimensional non-connected cylinders reduces to the case of connected ones. Indeed, if $\Omega$ is a union of segments, $\Omega_L$ is the union of rectangles with sides parallel to the axes, that have the same height and as bases the segments defining $\Omega$. Trivially, $h(\Omega_L)$ equals the Cheeger constant of the rectangle with the largest base, that is, the cylinder built on the largest segment defining $\Omega$, which surely exists, since $\Omega$ has finite measure.
\end{rem}

\begin{rem}
In~\cite{KLV19} the authors consider unbounded waveguides, that is, roughly speaking cylinders whose spine is the image of a generic unbounded curve $\gamma$ rather than a straight line. In~\cite[Remark~1]{KLV19} they essentially prove the upper bound~\eqref{eq:estimate_h} for the bounded waveguides $\gamma([0,L])\oplus B_r$ (topped with two half-balls), while they give a weaker lower bound independent of the length $L$, see~\cite[Theorem~1]{KLV19}.
\end{rem}

\section{Application}\label{sec:appl}

In this last section, we exploit the cylindrical estimates of \cref{thm:stime_h} to prove some properties on the shape functional $\Fp[\,\cdot\,]$ defined as
\[
\Fp[E] 
:= \displaystyle{\frac{\lambda^{\sfrac 1p}_p(E)}{h(E)}}\,,
\]
for $p\in (1, +\infty]$ with the convention that for $p=+\infty$ we let $\lambda^{\sfrac {1}{p}}_{p}(E) = \rho(E)^{-1}$, where this latter denotes the \emph{inradius} of the set $E$. Throughout the section we shall denote by $\mathbb{K}^N$ the class of convex subsets of $\mathbb{R}^N$. For the sake of notation, we also let
\[
\widetilde m_N := \inf_{E\subset\R^{N}} \Fp[E] \qquad \text{and} \qquad m_N := \inf_{E\in\mathbb{K}^{N}} \Fp[E]\,,
\]
without stressing the dependence on $p$ as this plays no role in the following.

\begin{thm}\label{thm:strict_decreasing}
For any fixed $p\in (1, +\infty]$, if there exist bounded minimizers of $\Fp[\,\cdot\,]$ among sets
\begin{itemize}
\item[(i)] in the Euclidean space $\mathbb{R}^N$, then $\widetilde m_{N+1} < \widetilde m_N$;
\item[(ii)] in the class of convex sets $\mathbb{K}^N$, then $m_{N+1} < m_N$.
\end{itemize}
\end{thm}

\begin{proof}
We only prove point~(i), as the proof of point~(ii) is completely analogous. Let us denote  by $\Omega$ a bounded minimizer of the functional in $\mathbb{R}^N$ and let us consider the cylinders $\Omega_L$ with cross-section $\Omega$ and height $L\ge 1$. 

Let us start with the case $1<p<+\infty$. First, we recall the upper bounds to $\lambda_p^{\sfrac 1p}(\Omega_L)$ proved in~\cite[Lemma~2.4]{BBP22}
\begin{alignat*}{3}
\lambda_p^{\sfrac 1p}(\Omega_L) &\le \left(\lambda_p^{\sfrac 2p}(\Omega)+ \frac{c}{L^2}\right)^{\frac 12}\,,
\qquad &&\text{if $p\ge 2$,}\\
\lambda_p^{\sfrac 1p}(\Omega_L) &\le \left(\lambda_p(\Omega)+ \frac{c}{L^p}\right)^{\frac 1p}\,,
&&\text{if $p\in(1, 2)$,}
\end{alignat*}
where $c=c(p)$ is explicit but not needed for our purposes. These imply that for $L\gg1$, i.e., for large enough values of $L$, one has
\[
\lambda^{\sfrac 1p}_p(\Omega_L) \le \lambda^{\sfrac 1p}_p(\Omega) + O\left( \frac{1}{L^{\min\{p, 2\}}}\right)\,.
\]
Combining this inequality with the lower bound to $h(\Omega_L)$ in~\eqref{eq:estimate_h} for $L\gg1$ have that
\begin{align*}
\widetilde m_{N+1} \le \frac{\lambda^{\sfrac 1p}_p(\Omega_L)}{h(\Omega_L)} 
\le \frac{\lambda^{\sfrac 1p}_p(\Omega)}{h(\Omega)}\cdot \frac{1+ O\left(\sfrac{1}{L^{\min\{p, 2\}}}\right))}{1+ \frac{ {c(\Omega)}}{Lh(\Omega)}} 
< \frac{\lambda^{\sfrac 1p}_p(\Omega)}{h^p(\Omega)} = \widetilde m_{N}\,.
\end{align*}

If otherwise $p=+\infty$, one has that for $L$ large enough $\rho(\Omega)=\rho(\Omega_L)$, and thus one can conclude in the same manner still owing to the lower bound in~\eqref{eq:estimate_h}.
\end{proof}

The above theorem becomes particularly useful when combined with the existence criterion~\cite[Theorem~3.6]{BBP22}, which was first devised in~\cite[Theorem~1.2]{Fto21} in the case $p=2$. We stress that the criterion only works when dealing with convex sets. For the sake of convenience, we recall it below along with a sketch of the proof, also highlighting how convexity plays a major role.

\begin{thm}[Theorem~3.6 of~\cite{BBP22}]\label{thm:ex_criterion}
For any fixed $p\in (1, +\infty]$, if $m_{N+1} < m_N$ holds, then there exists a bounded minimizer of $\Fp[\,\cdot\,]$ over $\mathbb{K}^{N+1}$.
\end{thm}

\begin{proof}[Sketch of the proof of \cref{thm:ex_criterion}]
First notice that, using the same argument of the proof of~\cref{thm:strict_decreasing} with the weaker lower bound~\eqref{eq:weaker} to $h(\Omega_L)$, one has that $m_{N+1} \le m_N$.

Second, the key observation behind the criterion is that, refer to~\cite[Proposition~3.5]{BBP22}, if a sequence $\{E_j\}_j$ of equimeasurable $(N+1)$-dimensional sets is such that the sequence of diameters $\{\diam E_j\}_j$ is unbounded, then $m_N \le \liminf_j \Fp[ E_j]$. Taking this for granted, if the strict inequality $m_{N+1} < m_N$ holds, one can rule out that minimizing sequence $\{E_j\}_j$ have unbounded diameters. Therefore, and here convexity matters, one can invoke the Blaschke Selection Principle and extract a subsequence converging in the Hausdorff metric to a bounded, convex set, which is easily shown to be a minimizer.

For the sake of completeness, we briefly sketch also the proof of the key observation, which also relies on the convexity of the sets $\{E_j\}_j$. 

Fixed any $j$, up to a translation and a rotation, one can assume that both the origin and the point $(0, \dots, 0, \diam E_j)$ belong to $\partial E_j$. We consider the section $\omega_j := E_j \cap \{\,x_{N+1}=t_j\,\}$, chosen as the section attaining
\[
\inf_{t\in [0, \diam E_j]}\lambda_p (E_j \cap \{\,x_{N+1}=t\,\})\,,
\]
which exists in virtue of the Hausdorff continuity of the sections $E_j \cap \{\,x_{N+1}=t\,\}$ in $\R^N$ (being $E_j$ convex), and the continuity of $\lambda_p$ with respect to the Hausdorff metric. We mention that, up to a reflection, one can also assume that 
\begin{equation}\label{eq:tj_big}
t_j \ge \frac{\diam E_j}{2}\,.
\end{equation}
Owing to the fact that the sections $E_j \cap \{\,x_{N+1}=0\,\}$ and $E_j \cap \{\,x_{N+1}=\diam E_j\,\}$ are empty, and owing to Fubini Theorem it can be seen that
\begin{equation}\label{eq:prima}
\lambda_p (E_j) \ge \lambda_p (\omega_j)\,,
\end{equation}
refer to~\cite[Lemma~3.3]{BBP22}.  Fixed a parameter $\alpha \in (0,1)$, one now considers the cylinders $C^\alpha_j \subset E_j$ of base $\alpha \omega_j$ and height $t_j(1-\alpha)$ contained in the cone given by the convex envelope of the origin with $\omega_j$, see also \cref{fig:idea_of_proof}. 
\begin{figure}
\begin{tikzpicture}
\clip (-1, -5) rectangle (9, 5); 
\begin{scope}[shift={(-.4125,0)}]
\draw[black, rounded corners=10mm] plot coordinates{(0,0) (2,2) (4,3) (7,3) (9, 0) (6.5, -3) (3,-3)}--cycle;
\filldraw[black!20, opacity=1] (5.5,3) -- (5.5,-3) -- (0.4125,0) -- (5.5,3);
\filldraw[black!40, opacity=1] (5.5, .5) -- (5.5, -.5) -- (1.2598, -.5) -- (1.2598, .5) -- (5.5,.5);
\node[] at (1.7,0.225) {$C^\alpha_j$};
\draw[thick] (5.5,3) -- (5.5,-3);
\node[] at (5.8,2.3) {$\omega_j$};
\end{scope}
\draw[black, ->] (0,-4) -- (0,4);
\draw[black, ->] (-1,0) -- (8.5,0);
\end{tikzpicture}
\caption{The cylinder $C^\alpha_j$ that is used to prove \cref{thm:ex_criterion}.}
\label{fig:idea_of_proof}
\end{figure}
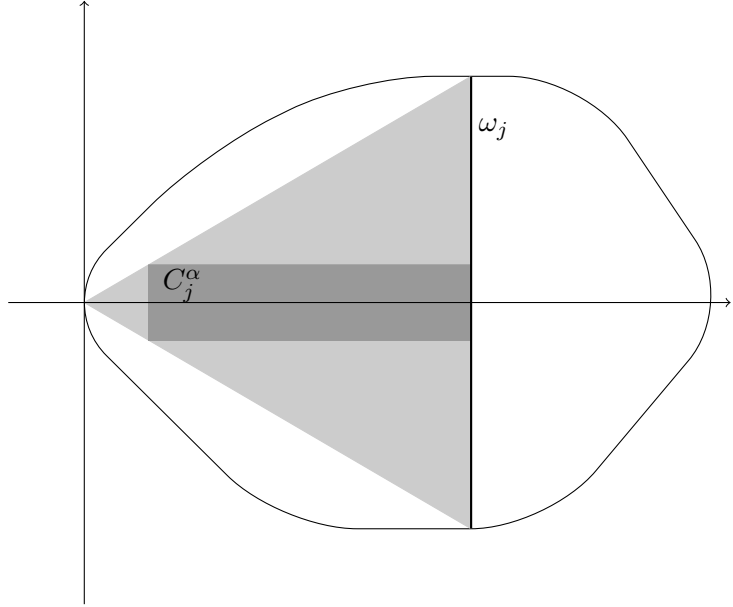
By the monotonicity of the Cheeger constant, we then have
\begin{equation}\label{eq:seconda}
h(C^\alpha_j) \ge h(E_j).
\end{equation}
Using now~\eqref{eq:prima} and~\eqref{eq:seconda}, multiplying and dividing by $h(\omega_j)$, using the scaling properties of the Cheeger constant, and owing to the upper estimate in~\eqref{eq:estimate_h} one has
\[
\Fp[E_j] = \frac{\lambda^{\sfrac 1p}_p (E_j) }{h(E_j)} \ge \alpha m_{N} \frac{h(\omega_j)}{h(\omega_j) + \frac{2\alpha}{(1-\alpha)t_j}}\,.
\]
Taking the inferior limit as $j\to+\infty$, using~\eqref{eq:tj_big}, and then letting $\alpha\to 1$, the claim follows.
\end{proof}

It is clear that by combining the above theorem with \cref{thm:strict_decreasing}(ii), it follows that existence of minimizers over $\mathbb{K}^M$ implies existence over $\mathbb{K}^N$ for any $N\ge M$. Hence, to conclude existence in all dimensions it would suffice to prove it for $N=1$. We show how the induction works and that it can start in the following theorem.

\begin{thm}\label{thm:ex_lp/h^p}
For any fixed $p\in (1, +\infty]$, the sequence $\{ m_N \}$ is strictly decreasing, and there exist bounded minimizers of $\Fp[\,\cdot\,]$ in $\mathbb{K}^N$ for any $N\in \N$.
\end{thm}

\begin{proof}
Assume that a minimizer exists in $\mathbb{K}^N$. Then \cref{thm:strict_decreasing}(ii) implies that the infimum over $\mathbb{K}^{N+1}$ is strictly less than that in $\mathbb{K}^N$. In turn, \cref{thm:ex_criterion} implies that a minimizer exists over $\mathbb{K}^{N+1}$. Therefore, the existence of minimizers in dimension $N=1$ would immediately imply both the strictly monotone decreasing behavior of the sequence and the existence of minimizers over $\mathbb{K}^{N}$ for all $N\in \N$. It is well-known that any interval minimizes the functional over $\mathbb{K}^1$, refer for instance to~\cite[Proposition~2.1]{BBP22}, thus the claim follows.
\end{proof}

\begin{rem}
So far \cref{thm:ex_lp/h^p} had only been proved in the $2$-dimensional case (for $p\ge 2$), and minimizers conjectured to exist in any dimension, and we here give a positive answer. We refer to~\cite[Proposition~5.2]{Par17} for $p=2$, where it is also conjectured that the minimum is attained for the square, and to~\cite[Theorem~3.8]{BBP22} for $p\ge 2$.
\end{rem}

It has already been observed in~\cite{BBP22, Fto21} that Cheeger's inequality is attained asymptotically as $N\to +\infty$. Indeed, it is easy to see that, fixed any $p$, a sequence of $N$-dimensional unit balls $\{B^N_1\}$ achieves the equality in the limit when $N\to +\infty$. In particular, for $p=2$, this follows from the explicit knowledge of the first eigenvalue of the $N$-dimensional ball and its asymptotic behavior, refer to~\cite{Tri49} and the computations carried out in~\cite[Theorem~1.2]{Fto21}. For general $p$ it follows from the estimates of~\cite[Lemma~2.3]{BBP22} and the computations carried out in~\cite[Theorem~2.6]{BBP22}, which give
\[
\frac{1}{p} \le m_N \le \Fp[B^N_1]\,\, \underset{\scriptscriptstyle N\to+\infty}{\sim}\,\, \frac{1}{p}.
\]
The results in~\cite{BBP22, Fto21} left open to the possibility that starting from some dimension $\bar{N}$, one had $m_N = \frac 1p$ for all $N\ge \bar N$. The strict monotonicity proved in \cref{thm:ex_lp/h^p} implies that this is not the case and thus Cheeger's inequality is attained \emph{only} asymptotically.

\begin{cor}
Cheeger's inequality $\Fp[E] \ge \frac 1p$, among convex sets $E\in\mathbb{K}^N$, is saturated \emph{only} asymptotically as $N\to +\infty$.
\end{cor}



\bibliographystyle{plainurl}

\bibliography{cilindri}

\begin{thebibliography}{10}

\bibitem{AFP00book}
L.~Ambrosio, N.~Fusco, and D.~Pallara.
\newblock {\em {F}unctions of {B}ounded {V}ariation and {F}ree {D}iscontinuity
  {P}roblems}.
\newblock Oxford Mathematical Monographs, 2000.

\bibitem{Avi97}
A.~Avinyo.
\newblock Isoperimetric constants and some lower bounds for the eigenvalues of
  the {$p$}-{L}aplacian.
\newblock In {\em Proceedings of the {S}econd {W}orld {C}ongress of {N}onlinear
  {A}nalysts, {P}art 1 ({A}thens, 1996)}, volume~30, pages 177--180, 1997.

\bibitem{Bra18}
L.~Brasco.
\newblock On principal frequencies and inradius in convex sets.
\newblock In {\em Bruno {P}ini {M}athematical {A}nalysis {S}eminar 2018},
  volume~9 of {\em Bruno Pini Math. Anal. Semin.}, pages 78--101. Univ.
  Bologna, Alma Mater Stud., Bologna, 2018.

\bibitem{BBP22}
L.~Briani, G.~Buttazzo, and F.~Prinari.
\newblock On a class of {C}heeger inequalities.
\newblock {\em Ann. Mat. Pura Appl.~(4)}, 202:657--678, 2023.

\bibitem{Che70}
J.~Cheeger.
\newblock A lower bound for the smallest eigenvalue of the {L}aplacian.
\newblock In {\em Problems in analysis ({P}apers dedicated to {S}alomon
  {B}ochner, 1969)}, pages 195--199. Princeton Univ. Press, Princeton, N.J.,
  1970.

\bibitem{CP24}
E.~Cinti and F.~Prinari.
\newblock On fractional {H}ardy-type inequalities in general open sets.
\newblock Preprint, 2024.
\newblock \href {https://arxiv.org/abs/2407.06568} {\path{arXiv:2407.06568}}.

\bibitem{FPSS22}
V.~Franceschi, A.~Pinamonti, G.~Saracco, and G.~Stefani.
\newblock The {C}heeger problem in abstract measure spaces.
\newblock {\em J. London Math. Soc.}, 119(1):e12840, 2024.

\bibitem{Fto21}
I.~Ftouhi.
\newblock On the {C}heeger inequality for convex sets.
\newblock {\em J. Math. Anal. Appl.}, 504(2), 2021.

\bibitem{FMP08}
N.~Fusco, F.~Maggi, and A.~Pratelli.
\newblock The sharp quantitative isoperimetric inequality.
\newblock {\em Ann. of Math. (2)}, 168(3):941--980, 2008.

\bibitem{Kaw16}
B.~Kawohl.
\newblock Two dimensions are easier.
\newblock {\em Arch. Math.}, 107(4):423--428, 2016.

\bibitem{KF03}
B.~Kawohl and V.~Fridman.
\newblock Isoperimetric estimates for the first eigenvalue of the
  {$p$}-{L}aplace operator and the {C}heeger constant.
\newblock {\em Comment. Math. Univ. Carolin.}, 44(4):659--667, 2003.

\bibitem{KL06}
B.~Kawohl and T.~Lachand-Robert.
\newblock Characterization of {C}heeger sets for convex subsets of the plane.
\newblock {\em Pacific J. Math.}, 225(1):103--118, 2006.

\bibitem{KLV19}
D.~Krej\v{c}i\v{r}\'ik, G.~P. Leonardi, and P.~Vlachopulos.
\newblock The {C}heeger constant of curved tubes.
\newblock {\em Arch. Math.}, 112:429--436, 2019.

\bibitem{KP11}
D.~Krej\v{c}i\v{r}\'ik and A.~Pratelli.
\newblock The {C}heeger constant of curved strips.
\newblock {\em Pacific J. Math.}, 254(2):309--333, 2011.

\bibitem{LW97}
L.~Lefton and D.~Wei.
\newblock Numerical approximation of the first eigenpair of the $p$-{L}aplacian
  using finite elements and the penalty method.
\newblock {\em Numer. Funct. Anal. Optim.}, 18(3-4):389--399, 1997.

\bibitem{Leo15}
G.~P. Leonardi.
\newblock {An overview on the Cheeger problem}.
\newblock In {\em New {T}rends in {S}hape {O}ptimization}, volume 166 of {\em
  Internat. Ser. Numer. Math.}, pages 117--139. Springer Int. Publ., 2015.

\bibitem{LP16}
G.~P. Leonardi and A.~Pratelli.
\newblock On the {C}heeger sets in strips and non-convex domains.
\newblock {\em Calc. Var. Partial Differential Equations}, 55(1):15, 2016.

\bibitem{Maz11book}
V.~Maz'ya.
\newblock {\em {S}obolev {S}paces with {A}pplications to {E}lliptic {P}artial
  {D}ifferential {E}quations}.
\newblock Number 342 in Grundlehren der mathematischen Wissenschaften.
  Springer-Verlag Berlin Heidelberg, 2nd edition, 2011.

\bibitem{Maz62}
V.~G. Maz'ya.
\newblock The negative spectrum of the higher-dimensional {S}chr\"{o}dinger
  operator.
\newblock {\em Dokl. Akad. Nauk SSSR}, 144:721--722, 1962.

\bibitem{Maz62-2}
V.~G. Maz'ya.
\newblock On the solvability of the {N}eumann problem.
\newblock {\em Dokl. Akad. Nauk SSSR}, 147:294--296, 1962.

\bibitem{Par11}
E.~Parini.
\newblock An introduction to the {C}heeger problem.
\newblock {\em Surv. Math. Appl.}, 6:9--21, 2011.

\bibitem{Par17}
E.~Parini.
\newblock Reverse {C}heeger inequality for planar convex sets.
\newblock {\em J. Convex Anal.}, 24(1):107--122, 2017.

\bibitem{Tri49}
F.~Tricomi.
\newblock Sulle funzioni di {B}ellel di ordine e argomento pressoch\`{e}
  uguali.
\newblock {\em Atti Accad. Sci. Torino Cl. Sci. Fis. Mat. Natur.}, 83:3--20,
  1949.

\end{thebibliography}

\end{document}